\documentclass[letterpaper, 10 pt, conference]{ieeeconf}  

\IEEEoverridecommandlockouts                              
\usepackage{amsfonts}
\usepackage{cellspace} 
\usepackage{diagbox} 
\usepackage{amsmath}
\usepackage{amssymb}
\usepackage{mathtools}
\usepackage{tikz}
\usetikzlibrary{chains}
\usepackage{listings}
\usepackage{subcaption}
\usepackage{float}
\usepackage{algorithm}
\usepackage[noend]{algorithmic}
\usepackage{dsfont}
\usepackage{hyperref}
\usepackage[most]{tcolorbox}
\usepackage{mathabx}
\usepackage{nicefrac}
\usepackage{algorithmic}

\allowdisplaybreaks

\DeclareMathOperator*{\argmin}{arg\,min}

\def\ddefloop#1{\ifx\ddefloop#1\else\ddef{#1}\expandafter\ddefloop\fi}

    \def\ddef#1{\expandafter\def\csname c#1\endcsname{\ensuremath{\mathcal{#1}}}}
    \ddefloop ABCDEFGHIJKLMNOPQRSTUVWXYZ\ddefloop

    \def\ddef#1{\expandafter\def\csname s#1\endcsname{\ensuremath{\mathsf{#1}}}}
    \ddefloop ABCDEFGHIJKLMNOPQRSTUVWXYZ\ddefloop

    \def\ddef#1{\expandafter\def\csname b#1\endcsname{\ensuremath{\mathbb{#1}}}}
    \ddefloop ABCDEFGHIJKLMNOPQRSTUVWXYZ\ddefloop
    
\newtheorem{thm}{Theorem}
\newtheorem{lem}[thm]{Lemma}

\newtheorem{assum}[thm]{Assumption}
\newtheorem{defn}[thm]{Definition}

\newtheorem{remark}{Remark}

\def\argmin{\operatornamewithlimits{arg\,min}}

\title{\LARGE \bf
Steady-state Based Approach to Online Non-stochastic Control
}

\author{Vijeth Hebbar, Spencer Hutchinson, Mahnoosh Alizadeh and C\'edric Langbort 
\thanks{Vijeth Hebbar and C\'edric Langbort acknowledge funding support from the ARO MURI grant W911NF-20-0252 (76582 NSMUR). Spencer Hutchinson and Mahnoosh Alizadeh acknowledge funding support from NSF grant \#2330154.}
\thanks{Vijeth Hebbar and C\'edric Langbort are with the Coordinated Science Lab, University of Illinois at Urbana–Champaign,          Urbana, IL 61801, USA.
         {\tt\small \{vhebbar2,langbort\}@illinois.edu}}
\thanks{Spencer Hutchinson and Mahnoosh Alizadeh are with the 
Department of Electrical and Computer Engineering, University of California, Santa Barbara, Santa Barbara, CA 93106 USA.
        {\tt\small \{shutchinson,alizadeh\}@ucsb.edu}} %
}

\begin{document}

\maketitle
\thispagestyle{empty}
\pagestyle{empty}

\begin{abstract}
    We study the problem of \textit{online non-stochastic control} (ONC), which is the control of a linear system under adversarial disturbances and adversarial cost functions, with the aim of minimizing the total cost incurred.
    A recent line of literature in ONC develops algorithms that enjoy sublinear regret with respect to a benchmark based on the set of steady-states that are attainable by a constant input.
    In this work, we extend this research direction by giving an algorithm that enjoys $\cO(\sqrt{T})$ regret with respect to a richer benchmark set, namely the set of steady-states attainable under an \emph{affine controller}.
    Since this benchmark substantially broadens the comparison class, it provides significantly 
    stronger performance guarantees.
    Our proposed algorithm combines a Follow-The-Perturbed-Leader-style online non-convex optimization approach with a batching method that maintains stability despite changing policies.
    Although our proposed algorithm requires solving non-convex subproblems, we show that an approximate solution to this subproblem is sufficient to ensure $\cO(\sqrt{T})$ regret.
    Furthermore, numerical experiments show that our algorithm enjoys lower total cost and similar computation to existing methods in certain settings.    
\end{abstract}

\section{Introduction}
Recent years have seen growing interest in problems that lie at the intersection of online learning and control, motivated by the need to make sequential decisions in dynamic and uncertain situations \cite{li2021online,shi_online_2020,jiang_online_2025,nonhoff_online_2024}. Within this broad interface, \emph{online non-stochastic control} (ONC) has emerged as a powerful paradigm for controlling systems in environments that may be unpredictable or even adversarial \cite{agarwal_online_2019,hazan_introduction_2025}. In its canonical form, ONC considers a learner tasked with controlling a \emph{fully-observable} linear time-invariant system when facing a sequence of unknown convex costs and disturbances, both of which may be picked adversarially. Variants of this framework have since been developed to address settings with partial observability \cite{simchowitz_improper_2020}, time-varying dynamics \cite{gradu2023adaptive} and other extensions \cite{hazan_introduction_2025}.

In this paper, however, we focus on the canonical state feedback setting and develop a novel online control algorithm for it. Specifically, we consider the following linear time-invariant system:
\begin{equation}
    \label{eq:dyn}
    x_{t+1} = A x_t + B u_t + w_t,
\end{equation}
where $x_t \in \bR^N$ is the \textit{state}, $u_t \in \bR^M$ is the \textit{input}, and $w_t \subseteq \bR^N$ is the \textit{disturbance}. 
At each time step, the learner selects an input $u_t$, the system evolves according to \eqref{eq:dyn}, and the learner incurs a convex cost $f_t(x_t)$.\footnote{For ease of exposition, we take the stage cost to depend only on the state. This simplifies the notation without affecting the core idea of our novel method. 
}  The objective is to choose the sequence of inputs $\{u_t\}$ so as to minimize the cumulative cost
\begin{equation} \label{eq:cumu-cost}
 \cJ = \sum_{t=1}^T f_t(x_t).
\end{equation}
The key challenge is that when choosing $u_t$, the learner does not know the current or future costs and disturbances. That is, the learner must act at time $t$ without prior knowledge of $\{f_\tau\}_{\tau \ge t}$ or $\{w_\tau\}_{\tau \ge t}$. As a result, one cannot hope to compute the truly optimal controller for \eqref{eq:cumu-cost} in advance. Instead, performance is evaluated through the notion of \emph{policy regret}, which compares the cumulative cost incurred by the learner against that of the best controller in a specified benchmark class, chosen in hindsight.

This setup was first considered in the seminal work of Agarwal et al. \cite{agarwal_online_2019}, who developed an online control algorithm that achieves sublinear \emph{regret} with respect to the performance of the best-in-hindsight \emph{linear feedback controller}. Since then, this benchmark has remained central in much of the ONC literature, owing to its natural relevance in many classical control settings. It has also continued to underpin a number of subsequent extensions of the framework, including settings with safety constraints \cite{jiang_online_2025,li2021online}, unknown dynamics \cite{cassel_rate-optimal_2022,minasyan_online_2022}, and other considerations  \cite{cassel2020bandit,agarwal_logarithmic_2019}. 

In our recent work \cite{hebbar_revisiting_2025}, however, we took a different viewpoint and argued that, when the learner faces general convex costs that are not necessarily quadratic, benchmarks based on linear feedback controllers can be less interpretable. Motivated by this, we proposed an alternative benchmark based on the set of steady states achievable by constant inputs. In a simplified setting where the system matrix $A$ is stable, this set takes the form
\begin{equation}
    \cX = \{x\in \bR^N | x = Ax + Bu,\; u\in \bR^M\} \label{eq:ss_set_OL}
\end{equation}
and then the the resulting notion of regret can be defined as
\begin{equation}
    \cR_{\cX}(T) = \sum_{t=1}^T f_t (x_t) - \min_{x\in \cX} \sum_{t=1}^T f_t (x). \label{eq:regret-ss-OL}
\end{equation}
This viewpoint yielded a computationally inexpensive online control method with sublinear regret guarantees, while also giving the benchmark a direct operational meaning for general convex costs. Effectively, it compares the learner against the performance of an idealized controller that instantly stabilizes the corresponding \emph{undisturbed} system at the best steady-state of the system in hindsight. Beyond its interpretability and simplicity, this approach also performed favorably in empirical simulations relative to the \emph{disturbance action controller} (DAC)-based approach of Agarwal et al. \cite{agarwal_online_2019}. 

Despite these positives, 
the steady-state set used to define the benchmark in \eqref{eq:regret-ss-OL}  may be overly restrictive. First, this set can be much smaller than the ambient state space; indeed, it should be clear from \eqref{eq:ss_set_OL} that its dimension is fundamentally limited by the dimension of the input, making the benchmark especially restrictive in underactuated settings. Second, it only considers steady states that can be maintained by constant inputs in the open-loop dynamics, and therefore excludes stationary points that become achievable once one allows for feedback.  These limitations naturally motivate the main question we study in the present paper: ``can this steady-state viewpoint be extended to a richer comparison class?''

We answer this question in the affirmative. Specifically, we enlarge the comparison class from steady-states achievable under constant-inputs to the richer set of steady-states achievable under affine stabilizing controllers of the form $u = -Kx+v$. 



\subsection{Other related literature} 

A closely related line of work is online feedback optimization (OFO), where one designs a controller so that the closed-loop system tracks the solution of a (possibly time-varying) optimization problem \cite{hauswirth_optimization_2024,menta_stability_2018,colombino_online_2018}. At a high level, our approach shares with this literature the idea of steering the system toward specific steady-states (be they generated as solutions of an optimization problem or not), and the resulting algorithms similarly rely on the stability of dynamics to ensure that the physical state can track these `target' states. However, the setting we study is fundamentally different in that the sequence of cost functions is not known a priori and may be chosen adversarially. Consequently, rather than analyzing convergence or tracking error to the optimizer of a known program, we evaluate performance through a meaningful notion of regret.

It is also worth noting that several works in the online non-stochastic control literature consider alternative notions of regret. A particularly important line of work studies \emph{dynamic regret}, where the performance of an online controller is compared against that of the optimal controller with full knowledge of future costs and disturbances. Such a benchmark is clearly stronger, but obtaining meaningful guarantees with respect to them typically requires additional problem structure, such as access to a finite horizon of future costs \cite{li_online_2019}, or restriction to quadratic costs \cite{karapetyan_online_2023}.  
Other works consider alternate notions of dynamic regret, for example by comparing against the sequence of pointwise optimal steady-states \cite{nonhoff_online_2021} or against a clairvoyant controller that is restricted to pick a linear feedback gain at every time \cite{zhou_safe_2023}. The resulting dynamic regret bounds in all these works are typically instance-dependent, often scaling with the total variation in the optimal trajectory. In contrast, our objective in this paper is to obtain instance-independent static regret guarantees for the learning algorithms we propose.

\subsection{Paper Roadmap} 
In Section \ref{sec:setup}, we formalize our novel regret notion where the benchmark set consists of steady-states achievable under stabilizing affine controllers. In Section \ref{sec:batching_algo}, we present an online control algorithm developed to achieve good performance with respect to this newly defined regret notion. In Section \ref{sec:regret-result}, we provide formal guarantees on the regret performance of the proposed algorithm. Finally, in Section \ref{sec:simulations} we present some simulation results comparing the empirical performance of our proposed approach vis-a-vis existing methods.    

\section{Problem Setup} \label{sec:setup}

We work with the discrete-time system outlined in \eqref{eq:dyn} where $A$ and $B$ are \emph{known} real matrices with appropriate dimensions. Consistent with the broader literature on ONC, we also assume that the pair $(A,B)$ is stabilizable. Additionally, for the remainder of the paper, we make the following simplifying assumption.

\begin{assum}[No disturbances] \label{ass:no-disturbance} 
    The disturbances $\{w_t\}$ acting on our system in \eqref{eq:dyn} satisfy $w_t=0$ for all $t$.
\end{assum}

At first glance, Assumption~\ref{ass:no-disturbance} may appear restrictive, especially since the ONC framework is designed to accommodate adversarial disturbances. However, for the fully-observable problems we consider here, this assumption is largely without loss of generality. As elaborated in Appendix \ref{app:ext-non-zero-costs}, the effect of disturbance on the state can be separated out by using a superposition argument, allowing the original problem to be recast in terms of an equivalent disturbance-free \emph{nominal} system with a modified sequence of cost functions. We therefore begin by presenting our algorithm and  regret results in this disturbance-free setting, where the core idea of our approach is most transparent. We will then re-introduce disturbances in the simulations using a minor modification to our approach, also  described in Appendix \ref{app:ext-non-zero-costs}.

At each time step $t \in [T]$, a \textit{learner} first chooses a control input $u_t$, and then the \emph{convex} cost function $f_t : \bR^N \rightarrow \bR$ and next state $x_{t+1}$ are revealed. The sequence of 
cost functions $\{ f_t \}_{t\in [T]}$ can be arbitrary (subject to the assumptions in the Section \ref{sec:assump}). However, we do assume that 
they are fixed in advance, and therefore do not depend on the inputs. This places our setup in the class of problems with an \emph{oblivious} adversary and \emph{full-feedback} on costs. 

\subsection{Notion of Regret}

In this paper, we evaluate the performance of an online control policy relative to the best-in-hindsight stable affine control law, i.e., controllers of the form $u_t = - K x_t + v$ with stabilizing $K$. More precisely, our notion of \emph{regret} compares the learner’s cumulative cost against the cost at the best steady-state achievable, in hindsight, by some controller from this class.

To formalize this notion of regret, we first introduce the notion of strongly-stabilizing controllers: 

\begin{defn}[Strongly stabilizing controllers \cite{agarwal_online_2019}]
\label{defn:strong_stab_K}
Fix $\kappa \geq 0$ and $\gamma \in (0,1]$. A matrix
$K \in \bR^{M\times N}$ is called $(\gamma,\kappa)$-strongly stabilizing if there exist matrices $H,J \in \bR^{N\times N}$ such that
\[
    A-BK = H^{-1}JH,
\]
with
\(
    \|J\| \leq 1-\gamma,
    \qquad
    \|H\|\,\|H^{-1}\| \leq \kappa.
\)
We denote by $\cK$ the set of all such controllers. 
\end{defn}
We use $\|\cdot\|$ to denote both the spectral norm of a matrix and the 2-norm of a vector. Note that any stabilizing controller is $(\kappa,\gamma)$-strongly stable for some $\gamma$ and $\kappa$ \cite{cohen_online_2018}. Thus, strong stability is merely a quantitative refinement of stability where the constants $\kappa,\gamma$ guarantee a specific rate of convergence for the closed-loop system $x_{t+1} = (A - B K) x_t$.

For a fixed controller $K\in\cK$, the set of steady states that can be maintained by affine control laws with offsets $v$ being picked from a \emph{convex} input set $\cU$ can be defined as 
\begin{equation}
    \cX(K) \triangleq \left\{ x \in \bR^N| x = (A-BK) x + Bv, v\in \cU \right\} \label{eq:ss_set_fix_K}.
\end{equation}
This is a direct generalization of the steady-state benchmark considered in \cite{hebbar_revisiting_2025}. Indeed, in the special case where the open-loop system $A$ is already stable, taking $K=0$ recovers the earlier set in \eqref{eq:ss_set_OL}. By allowing the feedback gain $K$ to vary over the class $\cK$, however, we obtain a substantially richer collection of attainable steady states. This motivates the following definition.

\begin{defn}[Closed Loop Steady State Manifold] \label{def:ss-set-cl}
        \begin{flalign*}
            \cZ & \triangleq \bigcup_{K\in \cK} \cX(K) \\
            & = \left\{ z \in \bR^N : z = (A - B K) z + B v, K \in \cK, v \in \cU \right\}
        \end{flalign*}
\end{defn}
In words, $\cZ$ is simply the set of all closed-loop steady states that are attainable under affine strongly stabilizing controllers. Note that the set $\cZ$ is guaranteed to be non-empty by our assumption on the stabilizablity of $(A,B)$. It is also worth making the following remark at this stage:
\begin{remark} \label{rem:non-convex}
    Since $\cU$ is convex, then the set $\cX(K)$ is convex for each $K\in \cK$. However, $\cZ$ may not be convex in general since it is obtained through an uncountable union of convex sets.  
\end{remark}

Finally, we define regret as the difference in cumulative cost between the realized states and the best point in $\cZ$,
\begin{equation}
    \cR_{\cZ}(T) \triangleq \sum_{t=1}^T f_t(x_t) - \min_{x \in \cZ} \sum_{t=1}^T f_t(x). \label{eq:regret_ss_CL}
\end{equation}

To re-iterate our goal in this paper, we wish to design an online control algorithm whose regret with respect to this benchmark grows sub-linearly with the horizon $T$. Before presenting such an algorithm, we introduce a few additional assumptions that are standard in the online control literature~\cite{hazan_introduction_2025}. 
\subsection{Assumptions} \label{sec:assump}
We first assume boundedness of the convex input set $\cU$. 

\begin{assum}
\label{ass:set_bound}
    \[\|v\|\leq U \;\;\text{ for all } \;\; v\in\cU.\] 
\end{assum}

Next, we assume that the gradients of the cost functions grow sub-linearly away from the origin.
Notably, this includes, but is not limited to, quadratic costs.

\begin{assum}[Smoothness of Costs]
\label{ass:grad_bound}
    It holds that $\| \nabla f_t(x) \| \leq L D$ for all $\| x \| \leq D$.
\end{assum}

Assumptions \ref{ass:set_bound} and \ref{ass:grad_bound} are applicable in all the results presented in the remainder of this paper. 

\section{Online Control Algorithm} \label{sec:batching_algo}
To address the richer benchmark introduced in \eqref{eq:regret_ss_CL}, we now need an online algorithm that can compete with stabilizing at the best point in $\cZ$ in hindsight. 
When competing against a benchmark consisting of steady-states achievable under constant inputs, our earlier work \cite{hebbar_revisiting_2025} employed an online projected gradient descent method to generate a sequence of target steady-states for the system to track. 
However, that approach does not carry over directly in the present setting. The key reason is that enlarging the benchmark from steady-states achievable by constant inputs to those achievable under affine strongly-stabilizing controllers introduces two new difficulties that were absent before. 

First, as highlighted in Remark \ref{rem:non-convex}, the set $\cZ$ is non-convex. This rules out a projection-based method which requires convexity of the decision set in order to be well-defined. In our earlier paper, the relevant steady-state set enjoyed the necessary convexity properties, but that is no longer the case once the feedback gain $K$ is allowed to vary. Second, each point $z$ in $\cZ$ is associated with a particular affine controller, and so to track these changing target steady-state the learner must also change the underlying feedback gain over time. As a result, the system no longer evolves under a fixed closed-loop dynamics, but under a \emph{switched} family of stable dynamics. Since arbitrarily switching among stable controllers does not preserve stability, one must be more careful in how such switching occurs.  

\begin{algorithm}[h]  
\caption{Batching approach for OLC}
\label{alg:batching_based}
\begin{algorithmic}[1]
    \REQUIRE Parameters $H,\eta,x_1$ 
    \STATE Sample $\sigma_i \stackrel{i.i.d}{\sim} \exp{(\eta)}$, Set $n=1$, Set $f_0(z) = \langle \sigma, z \rangle$ \label{line:init}
    \WHILE{ $n\leq \lceil\nicefrac{T}{H}\rceil$ }
        \STATE Pick $z^n = \bO_{\epsilon} \left(\sum_{\tau=0}^{(n-1)H} f_\tau(z), \cZ \right)$  \label{line:update_z} 
        \STATE Pick $K^n,v^n$ such that $z^n = (I-A+BK^n)^{-1} B v^n$. \label{lin:update_Ku} 
        \FOR{$t = (n-1)H+1, \dots, n H$}
            \STATE Pick inputs $u_{t} = -K^n x_{t} + v^n$. \label{lin:input-generation}
            \STATE Update state $x_{t+1} = A x_t + B u_t$, incur cost $f_t(x_t)$
            \IF{t=T}
               \STATE Terminate
            \ENDIF
        \ENDFOR
        \STATE n = n+1
    \ENDWHILE
\end{algorithmic}

\end{algorithm}

To address these two difficulties, we build on the Follow-the-Perturbed-Leader (FTPL) framework for online \emph{non-convex} optimization \cite{suggala2020online}. 
However, unlike OGD, FTPL does not provide a deterministic or high-probability guarantee on the magnitude of change in iterates. Rather it only bounds this magnitude in expectation, and therefore does not immediately ensure that the policy changes slow enough to maintain stability.
To address this, we introduce a batching scheme where the policy is only updated with FTPL every $H$ time steps. 
In order to effectively apply FTPL in this batching scheme, we treat the sum of the cost functions in each batch as a single cost function that is fed in to the FTPL algorithm.

Our proposed algorithm is stated precisely in Algorithm~\ref{alg:batching_based}. 
The algorithm divides the learning horizon into batches of length $H$, which are indexed by $n \in \{1,..., \lceil T/H \rceil \}$.
At the start of each batch, the learner selects a target steady-state $z^n \in \mathcal Z$ with an approximate FTPL update (Line \ref{line:update_z}), where $\bO_\epsilon$ refers to an $\epsilon$-approximate minimization oracle as specified in Definition~\ref{def:approx_oracle}.

\begin{defn}[Approximate Minimization Oracle]
\label{def:approx_oracle}
    Given a function $f$ and set $\cX$, the oracle output $x^* = \bO_\epsilon\big(f,\cX\big)$ satisfies $f(x^*) \leq \min_{x\in \cX} f(x) + \epsilon$.
\end{defn}

The reason we allow for an approximate minimization oracle, is because the feasible set $\cZ$ is non-convex and therefore it might be difficult to solve a minimization to optimality.
Note that the FTPL update in line \ref{line:update_z} also includes the random regularization function $f_0 (z) = \langle \sigma, z \rangle$ where $\sigma$ is an exponential random variable with parameter $\eta$, as defined in line \ref{line:init}.

Once $z^n$ is selected, the learner chooses a feedback gain $K^n \in \mathcal K$ and an offset $v^n \in \mathcal U$ such that $z^n$ is a steady-state of the closed-loop system under the affine controller $u = -K^n x + v^n$ (Line \ref{lin:update_Ku}). Such a pair exists by the definition of $\mathcal Z$. The controller $(K^n,v^n)$ is then held fixed for the duration of the batch.

At a high level, the batching parameter $H$ trades off two competing effects. Larger batches allow the closed-loop dynamics under a fixed controller to settle closer to their associated steady-state, reducing tracking error. Smaller batches, on the other hand, allow the learner to adapt more quickly to the observed sequence of costs. 

We choose $H$ to balance these effects. While we formalize this intuition by providing formal bounds on the regret incurred by Algorithm \ref{alg:batching_based} in the following section, for now, we state the following important consequence of our batching scheme:
\begin{lem}[Bounded State] \label{lem:bounded_x}
    When inputs are generated according to Algorithm \ref{alg:batching_based} with
    the batch size $H = \lceil\frac{\ln{(2\kappa)}}{-\ln{1-\gamma}}\rceil$, the state of the system evolving under \eqref{eq:dyn} in the absence of disturbances remains bounded. In particular, there exists a constant $D_x<\infty$ such that    \[\|x_t\|\leq D_x \quad \forall t.\]
\end{lem}

The proof of this result is omitted in the interest of space.
Note that $D$ may depend on the norm of the initial state $x_1$ and the bound $U$ on the input set. 

\begin{remark}[Extension to non-convex stage costs] \label{sec:apply-non-conx_f}
    It is worth noting that, since Algorithm \ref{alg:batching_based} employs a non-convex optimization oracle, the same approach is also applicable to ONC problems with non-convex stage costs. In particular, both Algorithm~\ref{alg:batching_based} and the regret result stated in the following section can be extended unchanged to non-convex costs $\{f_t\}$ by replacing Assumption \ref{ass:grad_bound} with a equivalent Lipschitzness criterion. We nevertheless present the method in the convex-cost setting, both because this is the standard setting in ONC and because convexity affords us access to simple and relatively efficient approximate oracles $\bO_\epsilon$. We employ one such oracle in the simulations in  Section~\ref{sec:simulations}.
\end{remark}

\section{Regret Result} \label{sec:regret-result}

In this section, we give the regret guarantees of our proposed algorithm.
In particular, Theorem \ref{thm:regret-bd-batch-algo} shows that the algorithm enjoys $\cO(\sqrt{T})$ expected regret with an appropriate choice of algorithm parameters $\epsilon, \eta, H$.

\begin{thm} \label{thm:regret-bd-batch-algo}
    Picking inputs according to Algorithm \ref{alg:batching_based} with the batch size $H = \lceil\frac{\ln{(2\kappa)}}{-\ln{1-\gamma}}\rceil$, an optimization oracle $\bO_\epsilon$ with accuracy parameter $\epsilon\in \cO(\nicefrac{1}{T})$ and with learning parameter \(
        \eta = {\left(5L\sqrt{2NT\left(1-\frac{\ln(2\kappa)}{\ln(1-\gamma)} + \frac{8\kappa^3}{\gamma^2}\right)}\right)^{-1}} 
    \)  
    guarantees the following regret bound:
    \begin{flalign*}
        &\bE\left[\cR_{\cZ}(T)\right] \\
        & \quad \leq  10 N L \bar D \sqrt{2 N T  {\left(1-\frac{\ln(2\kappa)}{\ln(1-\gamma)} + \frac{8\kappa^3}{\gamma^2}\right)}}  +  \cO(1) 
    \end{flalign*}
        
    where $\bar D = \max\{D_x,\nicefrac{\kappa}{\gamma} \|B\| U\}$. The expectation in the regret bound is taken over the randomization $\sigma$ employed in the Algorithm \ref{alg:batching_based}.  
\end{thm}
\begin{proof}(Sketch)
    We defer the detailed proof to Appendix~\ref{sec:proof}, and give only a sketch here. To this end, we begin by defining the batch cumulative cost $F_n$ of the $n$th batch:
    \begin{equation}
        F_n(z)
        \triangleq
        \sum_{h=1}^{H} f_{(n-1)H+h}(z).
    \label{eq:batched-costs}
    \end{equation}
    The batch cost $F_n$ is central to our analysis approach because our algorithm can be viewed as applying FTPL on the sequence of batch costs. Indeed, with $N_H \triangleq \frac{T}{H}$ as the number of batches (which we take as integral for simplicity of exposition\footnote{The remaining terms if any only contribute a $\cO(H)$ factor, which can be absorbed into the $\cO(1)$ term because our choice of $H$ does not scale with $T$.}), the regret can be re-written as,
    \begin{flalign}
        & \mathcal R_{\mathcal Z}(T) = \underbrace{
        \sum_{n=1}^{N_H}
        \bigl(F_n(z^n)-F_n(z^\star)\bigr)}_{(A)} \label{eq:batched-regret-decomp} \\[-8pt] 
        & +\underbrace{
        \sum_{n=1}^{N_H}\sum_{h=1}^{H}
        \bigl(
            f_{(n-1)H+h}(x_{(n-1)H+h})-f_{(n-1)H+h}(z^n)
        \bigr)}_{(B)}, \notag
    \end{flalign}
    where $z^\star \in \argmin_{z \in \mathcal Z} \sum_{t=1}^{T} f_t(z)$.
    Intuitively, term (A) is the regret of FTPL in terms of the batch costs, and term (B) captures the difference in cost between the realized states $x_t$ and the steady-states $z^n$ chosen in each batch. Term (A) can be viewed as the regret of an FTPL algorithm operating on the sequence of batched costs $F_n$ and its analysis closely follows the approach taken by Suggala and Netrapalli \cite{suggala2020online}. Bounding term (B), however, is a bit more involved. 

    The analysis of term (B) relies on bounding the distance between the chosen steady state point $z^n$ and the realized state $x_t$.
    Indeed, if we can show that distance between $z^n$ and $x_t$ is $\cO(1/\sqrt{T})$, and that all $f_t$ are Lipschitz, then it holds that term (B)  is $\cO(\sqrt{T})$. In order to get there, we first express the tracking error in terms of the switched closed-loop dynamics and the variation of the target sequence ${z_t}$ (Lemma \ref{lem:state-evolution}) as follows 
    \begin{equation*} 
    \begin{split}
        & x_{t+1}-z_{t+1}  = \bar A_t \bar A_{t-1} \dots \bar A_1 (x_1-z_{t+1})\\
        & \quad \;+\; \sum_{\tau=1}^{t} \bar A_t \bar A_{t-1}\cdots \bar A_{\tau+1}\,\bigl(I-\bar A_\tau\bigr)\,(z_\tau-z_{t+1})
    \end{split}
    \end{equation*}
    where $\bar{A}_t \triangleq A - B K^{n_t}$ and $n_t$  denotes the batch containing time $t$. Thus, the tracking error is driven by two quantities: products of the closed-loop matrices \(\bar A_t\), and the variation of the FTPL iterates \(\{z^n\}\). The latter can be bounded using a generalized FTPL stability argument, which shows that the distance between targets selected in different batches scales as \(\mathcal{O}(T^{-1/2})\) in expectation (Lemma~\ref{lem:generalized-stability-bd}).

    The remaining step is to show that the matrix products above decay geometrically across batches. This is precisely where the batching scheme and the choice of \(H\) are used. Since the controller is held fixed within each batch and each \(K^n\) is \((\gamma,\kappa)\)-strongly stabilizing, we have
\begin{flalign*}
\|\bar A_t \bar A_{t-1}\cdots \bar A_s\|
&\leq \prod_{n=n_s+1}^{n_t}\|(A-BK^n)^H\| \\
&\leq \bigl(\kappa(1-\gamma)^H\bigr)^{\,n_t-n_s}
\leq 2^{-(n_t-n_s)},
\end{flalign*}
where the final step follows from our choice of $H$. 
In other words, the influence of targets selected in earlier batches on the current tracking error decays exponentially fast. Combining this geometric decay with the \(\mathcal{O}(T^{-1/2})\) bound on the variation of the FTPL iterates yields
\(
\mathbb{E}\!\left[\|x_t-z_t\|\right]=\mathcal{O}(T^{-1/2}).
\)
This gives \(\mathbb{E}[(B)] = \mathcal{O}(\sqrt{T})\), and hence completes the proof sketch.
\end{proof}
  
Theorem~\ref{thm:regret-bd-batch-algo} shows that the steady-state viewpoint developed for online non-stochastic control in \cite{hebbar_revisiting_2025} can be extended beyond the benchmark based on constant inputs to accommodate affine controller based benchmarks. Moreover, even though the new benchmark is strictly stronger, the qualitative regret bound still enjoys the same rate of $\sqrt{T}$. However, it should be noted that the cost of the extension appears in the constants, which now depend more heavily on the stability parameters and the accuracy of the approximate optimization oracle.

This stronger benchmark, however, comes at a significant computational price. Unlike the simpler gradient-based procedure presented for the constant-input comparison class in our previous work \cite{hebbar_revisiting_2025}, the present approach relies on an oracle-based approach developed online non-convex optimization \cite{suggala2020online}. In particular, each batch update requires solving a non-convex optimization problem to recover an affine controller. At the same time, the ability to work with an approximate oracle is an important advantage of the approach, since the regret guarantee does not require these subproblems to be solved exactly. In the next section, we describe and implement one such approximate oracle in our simulations, and show that the resulting procedure remains practically viable despite the added complexity introduced by the richer benchmark.

\section{Simulations} \label{sec:simulations}
All our simulations are run with the following system matrices \vspace{-0.2cm}
\[A = \frac{1}{3.6}\begin{bmatrix}
    1 & 0.2 & 0 \\
    0 &  1 &  0.2 \\
    0.2 & 0 & 1 \\
\end{bmatrix} \text{ and } B = \begin{bmatrix}
    0 & 1 \\
    0 & 0 \\
    1 & 0
\end{bmatrix}.\]

\[f_t(x) = (x-c_t)^\intercal Q_t (x - c_t) \quad \forall \; 1\leq t \leq T\]
where the sequence of positive definite matrices $\{Q_t\}$ and `target' vectors $\{c_t\}$ are generated i.i.d from their respective distribution\footnote{{For code and more details - github.com/vijeth27/OnlineLinearControl-AffineBenchmark.git}}. 
Importantly, we re-introduce the disturbances at this stage and assume that the system evolves according to \eqref{eq:dyn} under non-zero disturbances. We do so because a direct comparison with the method of Agarwal et al. \cite{agarwal_online_2019} require working in the same disturbance-driven setting. We make precise the minor modification required to our algorithm to handle disturbances in Appendix \ref{app:ext-non-zero-costs}. In all our simulation, the disturbances $\{w_t\}$ are sampled i.i.d from a uniform distribution. 
We run both the proposed method (Algorithm \ref{alg:batching_based}), referred to as \textbf{BatchFTPL}, and the online control method of Agarwal et al. \cite{agarwal_online_2019}, referred to as \textbf{DAC} for $T=500$ time steps on the problem outlined above. We run $N=50$ iterations each with a newly sampled sequence of $\{Q_t\},\{c_t\}$ and $\{w_t\}$. 

\begin{figure}[ht]
     \centering
     \begin{subfigure}[b]{0.35\textwidth}
         \centering
         \includegraphics[width=\textwidth]{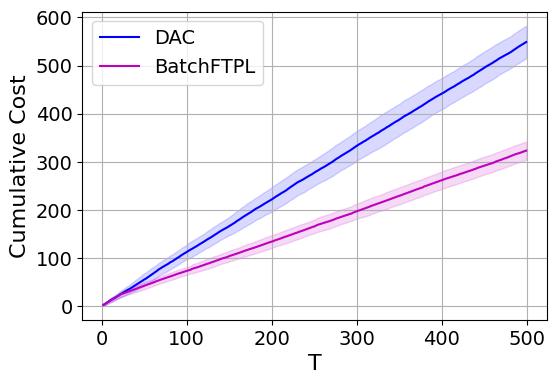}
         \caption{Cumulative cost incurred by both approaches}
         \label{fig:cumu-cost}
     \end{subfigure}
     \hfill
     \begin{subfigure}[b]{0.23\textwidth}
         \centering
         \includegraphics[width=\textwidth]{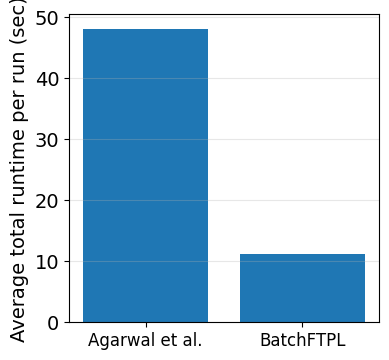}
         \caption{Total runtime averaged across $N$ iterations.}
         \label{fig:total-cost}
     \end{subfigure}
          \begin{subfigure}[b]{0.23\textwidth}
         \centering
         \includegraphics[width=\textwidth]{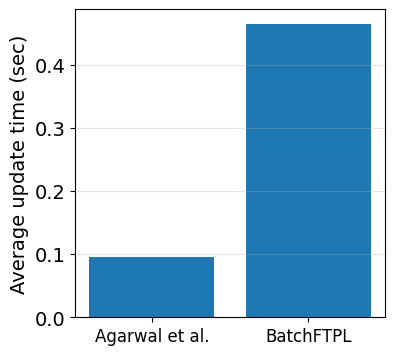}
         \caption{Average time for controller update}
         \label{fig:per-call-cost}
     \end{subfigure}
        \caption{\textbf{(Simulation Results)} Plots comparing both the performance and the computational cost of running the two approaches: \textbf{BatchFTPL} and \textbf{DAC}.}
        \label{fig:simu-results}
\end{figure}
A key ingredient in these simulations is the approximate oracle used in Line~\ref{line:update_z} of Algorithm~\ref{alg:batching_based}. Since $\cZ$ is generally non-convex, 
solving the corresponding optimization exactly is difficult. Instead, we approximate the set $\cK$ of stabilizing controllers used in defining the set $\cZ$ by a \emph{finite} bank of $N_{\cK}=100$ $(\gamma,\kappa)$-stabilizing controllers. For each controller $K$ in this bank, the oracle sub-problem reduces to a convex program over the corresponding steady-state set $\cX(K)$, and the overall oracle output is obtained by selecting the best candidate across the bank. While this discretized procedure is only approximate, it provides a practical implementation of the oracle required by the algorithm. 

Figure \ref{fig:cumu-cost} shows the mean and standard deviation of the cumulative cost incurred by the two approaches \textbf{BatchFTPL} and \textbf{DAC} across iterations. We observe that our proposed algorithm consistently outperforms \textbf{DAC} approach, indicating that the affine benchmark considered by \textbf{BatchFTPL} translates into a tangible performance improvement in the simulated environment. Figures~\ref{fig:total-cost} and \ref{fig:per-call-cost} compare the computational cost of the two implementations. As Figure~\ref{fig:per-call-cost} shows, a single oracle call in Line~\ref{line:update_z} is more expensive than one update of the DAC policy in the \textbf{DAC} approach, since even after discretizing $\cK$, each oracle evaluation scans over $N_{\cK}=100$ stabilizing gains and solves a convex subproblem for each. However, the oracle is invoked only once every $H$ steps. In our experiments, this lower update frequency more than offsets the larger per-call cost, yielding a lower total runtime than \textbf{DAC}, as shown in Figure~\ref{fig:total-cost}.

We emphasize that these experiments do not imply that the proposed approach is always computationally cheaper than a DAC based approach \cite{agarwal_online_2019}. The computational advantage observed here is implementation- and parameter-dependent. The simulations do show, however, that Algorithm~\ref{alg:batching_based}, even with a simple approximate oracle, can be competitive in both performance and runtime. In that sense, these results suggest that batching-based methods of this kind are a useful addition to the algorithmic toolbox for online non-stochastic control. 

\section{Conclusion} \label{sec:conclusion}

In this work, we extended the steady-state approach of online control introduced in our previous work \cite{hebbar_revisiting_2025} by enlarging the benchmark controller class from constant inputs to constant affine controllers. We showed that, despite the resulting non-convexity and the need to handle time-varying closed-loop dynamics, one can still obtain $\cO(\sqrt{T})$ regret. We note however, that the price of this stronger benchmark is computational. Our algorithm relies on an approximate non-convex optimization oracle at each controller update. Even so, our simulations showed that a simple and computationally inexpensive approximate oracle can already provide strong empirical performance, suggesting the practicality of our approach.

An important direction of future work involves exploring if this steady-state approach can be extended further. In particular, we considered static linear feedback in defining our benchmark set here, and a natural next step is to consider steady-states achievable under dynamic linear feedback controllers. 
More broadly, if this steady-state approach is to apply to more realistic systems, it will also be important to extend it to settings with safety constraints on states and inputs, as well as to problems with partial observability.

\appendix
\subsection{Extension to non-zero disturbances} \label{app:ext-non-zero-costs}
For clarity, the main development of the paper focuses on the disturbance-free setting. In the simulations, however, we re-introduce disturbances so as to allow a direct comparison with the disturbance-action based method of Agarwal et al. To do so, we use a state decomposition similar in spirit to the one employed in our earlier work \cite{hebbar_revisiting_2025}, but adapted to the present affine-controller setting.

Consider the system evolving according to \eqref{eq:dyn}. By the superposition principle, its state can be decomposed as \(x_t = \bar x_t + x_t^d\) 
where $x^d_t$ denotes the disturbance-driven part of the state, and $\bar x_t$ is the \emph{nominal} state evolution in the absence of disturbance.  Since the controller learned by Algorithm~\ref{alg:batching_based} changes from batch to batch, implementing it with feedback on the full state $x_t$ would lead to the disturbance-driven part also depending on the learned gains. To avoid this, we introduce a fixed strongly stabilizing gain $K_0 \in \cK$ chosen once in advance, and implement the control input in the form 
\begin{equation}
    u_t = -K^n \bar x_t - K_0 x^d_t +v^n \label{eq:decoupled-inputs}
\end{equation}
in the $n^{th}$ batch of Algorithm~\ref{alg:batching_based}. This form of input decouples the dynamics of the components as 
\[\bar x_t = (A-BK^n) \bar x_t + Bv^n,\;\; x^d_t = (A-BK_0) x^d_t + w_t,\]
and the disturbance-driven components evolves under fixed closed-loop dynamics. Moreover, this component remains bounded under bounded disturbances due to the strong stabilizing property of $K_0$. We may then define the translated cost functions 
    $g_t(\bar x) \triangleq f_t (\bar x + x^d_t).$
Since, $x^d_t$ is unaffected by the gains learned by Algorithm \ref{alg:batching_based}, the sequence $\{g_t\}$ thus defined is also independent of these choices. Hence, our simulations can be viewed as running Algorithm~\ref{alg:batching_based} on the nominal system under the translated costs $\{g_t\}$. 
For the regret plots reported in the simulations, both the total cost and the benchmark are evaluated in these nominal coordinates. 
However, because $f_t(x_t) = g_t(\bar x_t)$, these are identical to their values as incurred by the learner. 

\subsection{Some Helpful Lemmas}
We omit all proofs in the interest of space.
The following result is refinement of the well known bound \cite{cesa-bianchi_prediction_2006} on the regret of FTRL algorithms. 
\begin{thm} \label{thm:FTPL-general-bound}
    When $\{z^n\}_{n=1}^{N_H}$ is generated according to Line \ref{line:update_z},
    we can bound the expected regret of term (A) as
    \begin{flalign*}
        \bE[(A)] \leq & \sum_{n=1}^{N_H} \bE[F_n(z^n)-F_n(z^{n+1})] \\
        & + \bE[\langle \sigma,z^* - z^1\rangle] + \epsilon (N_H+1)
    \end{flalign*}
\end{thm}

\begin{lem}[Bounded $\cZ$]\label{lem:Z-bounded}
    \[
        \|z\|_\infty \le \frac{\kappa}{\gamma}\,\|B\|\,U \quad \forall \; z\in \cZ. 
    \]
\end{lem}

This result follows directly from Definition \ref{def:ss-set-cl}. 
Next, we have a bound on the distance between points picked according in Line \ref{line:update_z}. 
\begin{lem}[(Lemma 11, \cite{hebbar_responding_2024})] \label{lem:generalized-stability-bd}
    Let $n,n+h \in \{1,\dots,N_H\}$ for some $h>0$ then we have
    \begin{flalign} 
        \bE[\|z^{n+h}(\sigma)-z^{n}(\sigma)\|_1]  \leq    50\eta h L N^2 \bar D + 6 \frac{\epsilon}{hLH} \label{eq:stab_bound} 
    \end{flalign}
    where the expectation is taken over the distribution of $\sigma$.
\end{lem}

Finally, we state the following lemma which relates the realized system state to the sequence of past target states. 
\begin{lem}[Evolution of state] \label{lem:state-evolution}
Let $n_t$ refer to the batch that time step $t$ belongs to, and let $z_t \triangleq z^{n_t}$, $\bar{A}_t \triangleq A - B K^{n_t}$.
Then, the difference between the state $x_t$ of the system and the corresponding target state $z_t$ can be described as
\begin{equation}\label{eq:state-unrolled}
\begin{split}
x_{t+1}-z_{t+1} & = \Phi(t,1)\,(x_1-z_{t+1})\\
  \;+\; \sum_{\tau=1}^{t} & \Phi(t,\tau+1)\,\bigl(I-\bar A_\tau\bigr)\,(z_\tau-z_{t+1})
\end{split}
\end{equation}
where $\Phi(t,s) := \bar A_t \cdots \bar A_s$ when $(t\ge s)$, $\Phi(t,t+1):=I$. 
\end{lem}

\subsection{Proof of Theorem \ref{thm:regret-bd-batch-algo}}
\label{sec:proof}
We will pick up our proof from \eqref{eq:batched-regret-decomp} in the proof sketch where we split our regret into two terms $(A)$  and $(B)$. Recall that we are interested in bounding expected regret, so as we move forward, we will be bounding the expected values of these terms. Bounding the expected value of term $(A)$ follows the standard regret analysis approach for an FTPL algorithm operating on non-convex costs \cite{suggala2020online}, and is omitted in the interest of space.  It can be shown that 
\begin{flalign*}
     \bE[(A)]  \leq   50\eta HT L^2N^2 \bar{D} +  \bar{D} \frac{N}{\eta} + \epsilon (7N_H +1) 
\end{flalign*}


\textbf{Bounding term $(B)$:} By leveraging Jensen's inequality, Lemmas \ref{lem:bounded_x} and \ref{lem:Z-bounded}, and  Assumption \ref{ass:grad_bound} we get
\begin{equation} \label{eq:termB-expansion}
    \bE[(B)] \leq \sum_{n=1}^{N_H} \sum_{h=1}^H L\,\bE\!\left[\|x_{(n-1)H+h} - z^n\|\right].
\end{equation}

Applying Lemma~\ref{lem:state-evolution} with
$t=(n-1)H+h-1$, we obtain
\begin{align*}
& x_{(n-1)H+h}-z^n
= \\
& \Phi((n-1)H+h-1,1)\,(x_1-z^n) + \! \!\!\!\!
\sum_{\tau=1}^{(n-1)H+h-1}
 \! \!\!\! \!\\
& \Phi((n-1)H+h-1,\tau+1)\,(I-\bar A_\tau)\,(z_\tau-z^n).
\end{align*}
Taking norms, expectations, and using the triangle inequality yields
\begin{align*}
& \bE\!\left[\|x_{(n-1)H+h}-z^n\|\right] \\
&\le
\bE\!\left[\|\Phi((n-1)H+h-1,1)\|\,\|x_1-z^n\|\right] +  \!\! \!\!\!\sum_{\tau=1}^{(n-1)H+h-1}\! \!\!\!\!\!  \\
 &\bE\!\left[
\|\Phi((n-1)H+h-1,\tau+1)\|\,\|I-\bar A_\tau\|\,\|z_\tau-z^n\|
\right].
\end{align*}
Now, since $z_\tau=z^n$ for all $\tau=(n-1)H+1,\dots,(n-1)H+h-1$, the terms within the current batch vanish. Re-indexing the remaining terms batch-wise gives us 
\begin{align*}
 & \bE\!\left[\|x_{(n-1)H+h}-z^n\|\right] \le \\
 & \bE\!\left[\|\Phi((n-1)H+h-1,1)\|\,\|x_1-z^n\|\right] + \sum_{m=1}^{n-1}\sum_{h'=1}^{H} \\
 & \bE\!\big[
\|\Phi((n-1)H+h-1,(m-1)H+h'+1)\|\,\\ 
 & \qquad \|I-\underline{A}_m\|\,\|z^m-z^n\|
\big]
\end{align*}
where we used the notation 
    $\bar A_{(n-1) H + h}  = (A-B K^n) \triangleq \underline{A}_n \quad $ for all $  h \in [H], n\in [N_H].$ 
Then, by using the fact that
\begin{align}
    &\Phi\big((n-1)H+h,(m-1)H+h'\big) \notag \\ 
    & = \begin{cases}
        \underline{A}_n^{\,h-h'+1}, & n=m,\ h\ge h',\\[2mm]
        \underline{A}_n^{\,h}\left(\prod_{j=m+1}^{n-1}\underline{A}_j^{\,H}\right)\underline A_m^{\,H-h'+1}, & n>m,
    \end{cases} \label{eq:BatchedStateTransitionDiff} 
\end{align}
 the transition matrix can be decomposed as
\(
\Phi((n-1)H+h-1,(m-1)H+h'+1)
=
\underline A_n^{\,h-1}
\left(\prod_{j=m+1}^{n-1}\underline A_j^{\,H}\right)
\underline A_m^{\,H-h'}.
\)
Hence,
\begin{align*}
& \bE\left[\|x_{(n-1)H+h}-z^n\|\right] \\
&\le
\bE\left[
\left\|
\underline A_n^{h-1}\prod_{j=1}^{n-1}\underline A_j^{,H}
\right\|
\||x_1-z^n\||
\right] + \sum_{m=1}^{n-1}\sum_{h'=1}^{H} \\
& 
\bE\!\Bigg[
\|\underline A_n^{\,h-1}\|
\left\|\prod_{j=m+1}^{n-1}\underline A_j^{\,H}\right\|
\|\underline A_m^{\,H-h'}\|
\|I-\underline A_m\|\,\|z^m-z^n\|
\Bigg] \\
&\qquad\|I-\underline A_m\|\,\|z^m-z^n\| \Bigg] \\
&  \stackrel{(a)}{\le} \kappa^n (1-\gamma)^{(n-1)H+h-1} \bE\left[  
\||x_1-z^n\||
\right] +\sum_{m=1}^{n-1}\sum_{h'=1}^{H} \\
 & 
\kappa^{n-m+1} (1-\gamma)^{(n-m)H+h-h'-1} \bE\Bigg[
\|I-\underline A_m\|\,\|z^m-z^n\|
\Bigg]
\end{align*}
where inequality $(a)$ follows from strong stability of the closed loop matrices. Noting that $I-\underline A_m$ can be expanded as $H_m (I-J_m) H_m^{-1}$ for some $H_m$ and $J_m$, it follows that 
\(\|I-\underline A_m\| \leq \kappa \|I-J_m\| \leq 2 \kappa.\)
It is also easy to obtain the following bound
\(\sum_{h'=1}^{H} (1-\gamma)^{-h'-1} \leq \frac{1}{\gamma(1-\gamma)^{H+1}}.\)
Using these two bounds gives us,
\begin{flalign*}
    \bE\left[\|x_{(n-1)H+h}-z^n\|\right] \leq & (1-\gamma)^{h-1-H} \bigg( \nu^{n}  \bE\left[ \|x_1-z^n\| \right] \\
    +&   \frac{2\kappa^2}{\gamma} \sum_{m=1}^{n-1}
\nu^{n-m} \bE\left[\|z^m-z^n\|
\right] \bigg)
\end{flalign*}
where we define the quantity $\nu = \kappa (1-\gamma)^H$. Next, by invoking Lemma \ref{lem:generalized-stability-bd} on the norm of $\|z^m-z^n\|$ we have
\begin{flalign*}
    & \bE\left[\|x_{(n-1)H+h}-z^n\|\right] \leq  \\
    & \frac{1}{(1-\gamma)^{H+1-h}} \bigg( \nu^{n}  \bE\left[ \|x_1-z^n\| \right] + \\
    &  \frac{2\kappa^2}{\gamma} \sum_{m=1}^{n-1}
\nu^{n-m} \left( 50\eta (n-m) HL N^2 Z +  \frac{6\epsilon}{(n-m) H L} \right)\bigg). 
\end{flalign*}
Under the assumption that $\nu<1$ the second term includes two summation series which can be bound as follows
\(\sum_{m=1}^{n-1} \nu^{n-m}(n-m) \leq  \frac{\nu}{(1-\nu)^2}  \text{ and } \sum_{m=1}^{n-1} \frac{\nu^{n-m}}{(n-m)}  \leq -\ln{(1-\nu)}.\) Using these two summation bounds we get 
\begin{flalign*}
    & \bE\left[\|x_{(n-1)H+h}-z^n\|\right] \\
    & \leq (1-\gamma)^{h-1-H} \left( \nu^{n}  \bE\bigg[ \|x_1-z^n\| \right]  \\
    & \quad + \frac{2\kappa^2}{\gamma} \left( \eta \frac{50\nu HL N^2 Z}{(1-\nu)^2}  - \epsilon  \frac{6\ln{(1-\nu)}}{H L} \right)\bigg). 
\end{flalign*}

Noting $\|x_1 - z^n\| \leq \|x_1\| + \|z^n\| \leq 2\bar D$, substituting this bound back into to \eqref{eq:termB-expansion} and by using the fact that $(1-\gamma)^{-H} \nu = \kappa$, we have 
\begin{flalign*}
     & \bE[(B)] 
     \leq   L \sum_{n=1}^{N_H} \sum_{h=1}^H (1-\gamma)^{h-1} \\
     & \left( 2 \kappa \nu^{n-1}  \bar D + \frac{2\kappa^3}{\gamma} \left( \eta \frac{50 HL N^2 Z}{(1-\nu)^2}  - \epsilon  \frac{6\ln{(1-\nu)}}{H \nu L} \right)\right)  \stackrel{(a)}{\leq} \\
     & \frac{L}{\gamma} \sum_{n=1}^{N_H} \left( 2\kappa \nu^{n-1} \bar D + \frac{2\kappa^3}{\gamma} \left( \eta \frac{50 HL N^2 Z}{(1-\nu)^2}  - \epsilon  \frac{6\ln{(1-\nu)}}{H \nu L} \right)\right)  \\
     & \stackrel{(b)}{\leq}  \frac{L}{\gamma}  \left( \frac{2\kappa\bar D}{1-\nu}  + \frac{2\kappa^3}{\gamma} \left( \eta \frac{50 TL N^2 Z}{(1-\nu)^2}  - \epsilon  \frac{6N_H\ln{(1-\nu)}}{H \nu L} \right)\right)  
\end{flalign*}
where in inequality $(a)$ we used $\sum_{h=1}^H (1-\gamma)^{h-1} \leq \sum_{h=1}^\infty (1-\gamma)^{h-1} = \frac{1}{\gamma}$.  In inequality $(b)$ we used a similar bound for terms with $\nu$, and the fact that $N_HH = T$. 

\textbf{Putting terms $(A)$ and $(B)$ together:} 
By combining the bound for term $(A)$ and $(B)$ and re-arranging we get
\begin{flalign*}
    & \bE[\cR_{\cZ}(T)] \leq 50 \eta N^2TL^2 \bar D \left(H + \frac{2\kappa^3}{\gamma^2 (1-\nu)^2}\right) + \bar D \frac{N}{\eta} + \\
    & \frac{2\kappa L \bar D }{\gamma (1-\nu)} +  \epsilon\left( 7N_H + 1 - \frac{12N_
    H\kappa^3 \ln(1-\nu)}{H\gamma^2 \nu}\right)
\end{flalign*}
Our choice of $H$ in the theorem statement gives us $\nu=\frac{1}{2}$. Substituting this, along with the stated choices of $\epsilon$ and $\eta$  completes the proof. \hfill $\square$
 
%
\subsection*{Acknowledgements}
{\footnotesize An AI system (ChatGPT) was used to support the generation of simulation results presented in Section \ref{sec:simulations}. In particular, it helped in the creation of various helper functions used to implement Algorithm \ref{alg:batching_based}.  }

\bibliographystyle{IEEEtran}
\bibliography{onlineLinCont}

\end{document}